\documentclass{amsart}
\usepackage{setspace}
\usepackage{a4}
\usepackage{amsthm}
\usepackage{latexsym}
\usepackage{amsfonts}
\usepackage{graphicx}
\usepackage{textcomp}
\usepackage{cite}
\usepackage{enumerate}
\usepackage{amssymb}
\usepackage{hyperref}
\usepackage{amsmath}
\usepackage{tikz}
\usepackage[mathscr]{euscript}
\usepackage{mathtools}
\newtheorem{theorem}{Theorem}[section]

\newtheorem{corollary}[theorem] {Corollary}
\newtheorem{definition}[theorem]{Definition}

\title{This is the title}
\usepackage{fancyhdr}

\begin{document}
\hrule\hrule\hrule\hrule\hrule
\vspace{0.3cm}	
\begin{center}
{\bf{NONLINEAR MACCONE-PATI  UNCERTAINTY PRINCIPLE}}\\
\vspace{0.3cm}
\hrule\hrule\hrule\hrule\hrule
\vspace{0.3cm}
\textbf{K. MAHESH KRISHNA}\\
Post Doctoral Fellow \\
Statistics and Mathematics Unit\\
Indian Statistical Institute, Bangalore Centre\\
Karnataka 560 059, India\\
Email:  kmaheshak@gmail.com\\

Date: \today
\end{center}

\hrule\hrule
\vspace{0.5cm}
\textbf{Abstract}:  We show that one of the two important uncertainty principles derived by Maccone and Pati \textit{[Phys. Rev. Lett., 2014]} can be derived for arbitrary maps defined  on subsets of $\mathcal{L}^p$ spaces for $1< p<\infty$. Our main tool is the Clarkson inequalities. We also derive a nonlinear uncertainty principle for  weak parallelogram spaces  and Type-p Banach spaces.



\hrule

\hrule
\section{Introduction}
Let $\mathcal{H}$ be a complex Hilbert space and $A$ be a possibly unbounded self-adjoint operator defined on domain $\mathcal{D}(A)\subseteq \mathcal{H}$. For $h \in \mathcal{D}(A)$ with $\|h\|=1$, define the uncertainty of $A$ at the point $h$ as 
\begin{align*}
	\Delta _h(A)\coloneqq \|Ah-\langle Ah, h \rangle h \|=\sqrt{\|Ah\|^2-\langle Ah, h \rangle^2}. 
\end{align*}
In 1929, Robertson \cite{ROBERTSON} derived the following mathematical form of the uncertainty principle of Heisenberg derived in 1927 \cite{HEISENBERG}. Recall that for two operators $A:	\mathcal{D}(A)\to  \mathcal{H}$ and $B:	\mathcal{D}(B)\to  \mathcal{H}$, we define $[A,B] \coloneqq AB-BA$ and $\{A,B\}\coloneqq AB+BA$.
\begin{theorem} \cite{ROBERTSON, HEISENBERG, VONNEUMANNBOOK, DEBNATHMIKUSINSKI}  (\textbf{Heisenberg-Robertson Uncertainty Principle})
Let  $A:	\mathcal{D}(A)\to  \mathcal{H}$ and $B:	\mathcal{D}(B)\to  \mathcal{H}$  be self-adjoint operators. Then for all $h \in \mathcal{D}(AB)\cap  \mathcal{D}(BA)$ with $\|h\|=1$, we have 
\begin{align}\label{HR}
 \frac{1}{2} \left(\Delta _h(A)^2+	\Delta _h(B)^2\right)\geq \frac{1}{4} \left(\Delta _h(A)+	\Delta _h(B)\right)^2 \geq  \Delta _h(A)	\Delta _h(B)   \geq  \frac{1}{2}|\langle [A,B]h, h \rangle |.
\end{align}
\end{theorem}
In 1930,  Schrodinger improved Inequality (\ref{HR}) \cite{SCHRODINGER}. 
\begin{theorem} \cite{SCHRODINGER}
	(\textbf{Heisenberg-Robertson-Schrodinger  Uncertainty Principle})
	Let  $A:	\mathcal{D}(A)\to  \mathcal{H}$ and $B:	\mathcal{D}(B)\to  \mathcal{H}$  be self-adjoint operators. Then for all $h \in \mathcal{D}(AB)\cap  \mathcal{D}(BA)$ with $\|h\|=1$, we have 
	\begin{align*}
		\Delta _h(A)	\Delta _h(B)    \geq |\langle Ah, Bh \rangle-\langle Ah, h \rangle \langle Bh, h \rangle|
		=\frac{\sqrt{|\langle [A,B]h, h \rangle |^2+|\langle \{A,B\}h, h \rangle -2\langle Ah, h \rangle\langle Bh, h \rangle|^2}}{2}.
	\end{align*}	
\end{theorem}
A fundamental drawback of Inequality (\ref{HR}) is that if $h$ satisfies $ABh=BAh$, then the right side is zero. In 2014,  Maccone and Pati derived the following two uncertainty principles which overturned this  problem \cite{MACCONEPATI}.
\begin{theorem}  \cite{MACCONEPATI} (\textbf{Maccone-Pati Uncertainty Principle}) \label{MP}
Let  $A:	\mathcal{D}(A)\to  \mathcal{H}$ and $B:	\mathcal{D}(B)\to  \mathcal{H}$  be self-adjoint operators. Then for all $h \in \mathcal{D}(A)\cap  \mathcal{D}(B)$ with $\|h\|=1$, we have 
\begin{align*}
	\Delta _h(A)^2+	\Delta _h(B)^2 \geq      \frac{1}{2} \left(|\langle (A+B)h, k \rangle|^2+|\langle (A-B)h, k \rangle|^2\right), \quad \forall k \in \mathcal{H} \text{ satisfying } \|k\|=1, \langle h, k \rangle =0.
\end{align*}		
\end{theorem}
\begin{theorem} \cite{MACCONEPATI} (\textbf{Maccone-Pati Uncertainty Principle})
Let  $A:	\mathcal{D}(A)\to  \mathcal{H}$ and $B:	\mathcal{D}(B)\to  \mathcal{H}$  be self-adjoint operators. Then for all $h \in \mathcal{D}(A)\cap  \mathcal{D}(B)$ with $\|h\|=1$, we have 
\begin{align*}
	\Delta _h(A)^2+	\Delta _h(B)^2 \geq      -i\langle [A,B]h, h \rangle +|\langle (A+iB)h, k \rangle|^2, \quad \forall k \in \mathcal{H} \text{ satisfying } \|k\|=1, \langle h, k \rangle =0.
\end{align*}	
\end{theorem}
In this note, we show that Theorem \ref{MP} can be generalized  even for arbitrary maps on Lebesgue  spaces using Clarkson inequalities. We also derive uncertainty principle for  Banach spaces satisfying weak parallelogram law and  Type-p Banach spaces. 

Our main motivation comes from the sentence `The first proof, based on the parallelogram law, was communicated to us by an anonymous referee, while the second (independent) proof was our original argument'' given in  \cite{MACCONEPATI}. Note that Clarkson inequalities are generalizations of Jordan-von Neumann parallelogram law in Hilbert space \cite{JORDANVONNEUMANN}.

\section{Nonlinear Maccone-Pati Uncertainty Principle}
We first define the uncertainty for maps on Lebesgue  spaces. Let $\mathcal{M}\subseteq \mathcal{L}^p(\Omega, \mu)$ be a subset and $A:\mathcal{M}\to \mathcal{L}^p(\Omega, \mu)$ be a map (need not be linear or Lipschitz). Given $f\in \mathcal{M}$ and $a \in \mathbb{C}$, we define the uncertainty at $f$ relative to $a$ as 
\begin{align*}
	\Delta _f(A, a)\coloneqq \|Af-af \|_p. 
\end{align*}
To derive our first uncertainty principle we need the following breakthrough inequalities of Clarkson.
\begin{theorem} \cite{CLARKSON, GARLING, RAMASWAMY} (\textbf{Clarkson Inequalities})\label{C} 
Let  $(\Omega, \mu)$ be a measure space.
\begin{enumerate}[\upshape(i)]
	\item Let $2\leq p <\infty$. Then 
\begin{align*}
	\|f\|_p^p+\|g\|_p^p\geq \frac{1}{2^{p-1}} \left(\|f+g\|_p^p+\|f-g\|_p^p\right), \quad \forall f, g \in \mathcal{L}^p(\Omega, \mu).
\end{align*}
\item Let $1< p \leq 2$. Then 
\begin{align*}
	\|f\|_p^p+\|g\|_p^p\geq \frac{1}{2} \left(\|f+g\|_p^p+\|f-g\|_p^p\right), \quad \forall f, g \in \mathcal{L}^p(\Omega, \mu).
\end{align*}
\item  Let $2\leq p <\infty$ and $q$ be the conjugate index of $p$. Then 
\begin{align*}
	\|f\|_p^q+\|g\|_p^q\geq \left(\frac{1}{2} \left(\|f+g\|_p^p+\|f-g\|_p^p\right)\right)^\frac{1}{p-1}, \quad \forall f, g \in \mathcal{L}^p(\Omega, \mu).
\end{align*}
\item  Let $1<p\leq 2$ and $q$ be the conjugate index of $p$. Then 
\begin{align*}
	\|f\|_p^p+\|g\|_p^p\geq \left(\frac{1}{2} \left(\|f+g\|_p^q+\|f-g\|_p^q\right)\right)^\frac{1}{q-1}, \quad \forall f, g \in \mathcal{L}^p(\Omega, \mu).
\end{align*}
\item Let $1<p<\infty$ and $q$ be the conjugate index of $p$. Let $1<r\leq \min\{p,q\}$ and $s$ be the conjugate index of $r$. Then 
\begin{align*}
&	\|f\|_p^r+\|g\|_p^r\geq \left(\frac{1}{2} \left(\|f+g\|_p^s+\|f-g\|_p^s\right)\right)^\frac{1}{s-1}, \\
	&	\|f\|_p^s+\|g\|_p^s\geq \frac{1}{2^{s-1}} \left(\|f+g\|_p^s+\|f-g\|_p^s\right), \\
&			\|f\|_p^r+\|g\|_p^r\geq \frac{1}{2} \left(\|f+g\|_p^r+\|f-g\|_p^r\right), 
	\quad \forall f, g \in \mathcal{L}^p(\Omega, \mu).
\end{align*}
\end{enumerate}	
\end{theorem}
\begin{theorem} (\textbf{Nonlinear Maccone-Pati Uncertainty Principle}) \label{NMP}
	Let $(\Omega, \mu)$ be a measure space. Let 	$\mathcal{M}, \mathcal{N}\subseteq \mathcal{L}^p(\Omega, \mu)$ be  subsets and $A:\mathcal{M}\to \mathcal{L}^p(\Omega, \mu)$, $B:\mathcal{N}\to \mathcal{L}^p(\Omega, \mu)$ be  maps. Let  $f\in \mathcal{M} \cap \mathcal{N}$  and $a, b \in \mathbb{C}$.
\begin{enumerate}[\upshape(i)]
\item Let $2\leq p <\infty$.   Then 
\begin{align*}
	\Delta _f(A, a)^p+	\Delta _f(B,b)^p&\geq 	\frac{1}{2^{p-1}}\left(|\phi((A+B)f))|^p+|\phi((A-B)f)|^p\right), \quad \forall \phi  \in (\mathcal{L}^p(\Omega, \mu))^*\\
	& \quad \text{ satisfying } \|\phi\|\leq 1, \phi(f) =0.
\end{align*}
\item Let $1< p \leq 2$.   Then 
\begin{align*}
	\Delta _f(A, a)^p+	\Delta _f(B,b)^p&\geq 	\frac{1}{2}\left(|\phi((A+B)f))|^p+|\phi((A-B)f)|^p\right), \quad \forall \phi  \in (\mathcal{L}^p(\Omega, \mu))^*\\
	& \quad \text{ satisfying } \|\phi\|\leq 1, \phi(f) =0.
\end{align*}
\item  $2\leq p <\infty$ and $q$ be the conjugate index of $p$. Then 
\begin{align*}
	\Delta _f(A, a)^q+	\Delta _f(B,b)^q&\geq \left(	\frac{1}{2}\left(|\phi((A+B)f))|^p+|\phi((A-B)f)|^p\right)\right)^\frac{1}{p-1}, \quad \forall \phi  \in (\mathcal{L}^p(\Omega, \mu))^*\\
	& \quad \text{ satisfying } \|\phi\|\leq 1, \phi(f) =0.
\end{align*}
\item  $1<p \leq 2$ and $q$ be the conjugate index of $p$. Then 
\begin{align*}
	\Delta _f(A, a)^p+	\Delta _f(B,b)^p&\geq \left(	\frac{1}{2}\left(|\phi((A+B)f))|^q+|\phi((A-B)f)|^q\right)\right)^\frac{1}{q-1}, \quad \forall \phi  \in (\mathcal{L}^p(\Omega, \mu))^*\\
	& \quad \text{ satisfying } \|\phi\|\leq 1, \phi(f) =0.
\end{align*}
\item Let $1<p<\infty$ and $q$ be the conjugate index of $p$. Let $1<r\leq \min\{p,q\}$ and $s$ be the conjugate index of $r$. Then 
\begin{align*}
		\Delta _f(A, a)^r+	\Delta _f(B,b)^r&\geq \left(	\frac{1}{2}\left(|\phi((A+B)f))|^s+|\phi((A-B)f)|^s\right)\right)^\frac{1}{s-1}, \\
		\Delta _f(A, a)^s+	\Delta _f(B,b)^s&\geq 	\frac{1}{2^{s-1}}\left(|\phi((A+B)f))|^s+|\phi((A-B)f)|^s\right),\\
				\Delta _f(A, a)^r+	\Delta _f(B,b)^r&\geq 	\frac{1}{2}\left(|\phi((A+B)f))|^r+|\phi((A-B)f)|^r\right), \quad \forall \phi  \in (\mathcal{L}^p(\Omega, \mu))^*\\
		& \quad \text{ satisfying } \|\phi\|\leq 1, \phi(f) =0.
\end{align*}
\end{enumerate}
\end{theorem}
\begin{proof}
	We prove (i) and remaining are similar. Using (i)  in Theorem \ref{C}
	\begin{align*}
			\Delta _f(A, a)^p+	\Delta _f(B,b)^p&=\|Af-af \|_p^p+\|Bf-bf \|_p^p\\
			&\geq \frac{1}{2^{p-1}} \left(\|(Af-af)+(Bf-bf)\|_p^p+\|(Af-af)-(Bf-bf)\|_p^p\right)\\
			&=\frac{1}{2^{p-1}} \left(\|(A+B)f-(a+b)f\|_p^p+\|(A-B)f-(a-b)f\|_p^p\right)\\
			&\geq \frac{1}{2^{p-1}}\|\phi\|\left(\|(A+B)f-(a+b)f\|_p^p+\|(A-B)f-(a-b)f\|_p^p\right)\\
			&\geq \frac{1}{2^{p-1}}\left(|\phi((A+B)f)-\phi((a+b)f)|^p+|\phi((A-B)f)-\phi((a-b)f)|^p\right)\\
			&=\frac{1}{2^{p-1}}\left(|\phi((A+B)f))|^p+|\phi((A-B)f)|^p\right).
	\end{align*}
\end{proof}
We next note the following extension of Theorem \ref{NMP} (we state generalizations of only (i) and (ii) and others are similar).  Recall that  the collection of all Lipschitz functions $\psi:  \mathcal{L}^p(\Omega, \mu) \to \mathbb{C}$  satisfying  $\psi(0)=0$, denoted by $ \mathcal{L}^p(\Omega, \mu)^\# $ is a Banach space \cite{WEAVER} w.r.t. the Lipschitz norm 
\begin{align*}
	\|\psi\|_{\text{Lip}_0}\coloneqq \sup_{f, g \in \mathcal{L}^p(\Omega, \mu), f\neq g} \frac{|\psi(f)-\psi(g)|}{\|f-g\|_p}.
\end{align*}

\begin{corollary}
	Let $(\Omega, \mu)$ be a measure space. Let 	$\mathcal{M}, \mathcal{N}\subseteq \mathcal{L}^p(\Omega, \mu)$ be  subsets and $A:\mathcal{M}\to \mathcal{L}^p(\Omega, \mu)$, $B:\mathcal{N}\to \mathcal{L}^p(\Omega, \mu)$ be  maps. Let  $f\in \mathcal{M} \cap \mathcal{N}$  and $a, b \in \mathbb{C}$ be such that $a+b=1$.  
	\begin{enumerate}[\upshape(i)]
\item Let $2\leq p <\infty$.  Then 
\begin{align*}
	\Delta _f(A, a)^p+	\Delta _f(B,b)^p&\geq 	\frac{1}{2^{p-1}}\left(|\phi((A+B)f))|^p+|\phi((A-B)f)|^p\right), \quad \forall \phi  \in (\mathcal{L}^p(\Omega, \mu))^\#\\
	& \quad \text{ satisfying } \|\phi\|_{\text{Lip}_0}\leq 1, \phi(f) =0.
\end{align*}	
	\item Let $1< p \leq 2$.  Then 
	\begin{align*}
		\Delta _f(A, a)^p+	\Delta _f(B,b)^p&\geq 	\frac{1}{2}\left(|\phi((A+B)f))|^p+|\phi((A-B)f)|^p\right), \quad \forall \phi  \in (\mathcal{L}^p(\Omega, \mu))^\#\\
		& \quad \text{ satisfying } \|\phi\|_{\text{Lip}_0}\leq 1, \phi(f) =0.
	\end{align*}	
\end{enumerate}
\end{corollary}
In 1972, Bynum and Drew derived the following surprising result \cite{BYNUMDREW}. 
\begin{theorem}\cite{BYNUMDREW}\label{BYNUMDREW}
	For $1<p<2$,  
	\begin{align*}
		\|x+y\|_p^2+(p-1)\|x-y\|_p^2\leq 2\left(\|x\|_p^2+\|y\|_p^2\right), \quad \forall x, y \in \ell^p(\mathbb{N}).
	\end{align*}
\end{theorem}
Theorem \ref{BYNUMDREW} promoted the notion of parallelogram law spaces by Cheng and Ross \cite{CHENGROSS}.
\begin{definition}\cite{CHENGROSS}
Let $C>0$ and $1<p<\infty$. A Banach space $\mathcal{X}$ is said to satisfy lower p-weak parallelogram law with constant $ C$ if 
	\begin{align*}
	\|x+y\|^p+C\|x-y\|^p\leq 2^{p-1}\left(\|x\|^p+\|y\|^p\right), \quad \forall x, y \in \mathcal{X}.
\end{align*}
in this case, we write $\mathcal{X}$ is $p$-LWP(C).
\end{definition}
For parallelogram law spaces, by following a similar computation as in the proof of Theorem \ref{NMP} we get the following theorem.  
\begin{theorem}
	Let $\mathcal{X}$ be  $p$-LWP(C).  Let 	$\mathcal{M}, \mathcal{N}\subseteq \mathcal{X}$ be  subsets and $A:\mathcal{M}\to \mathcal{X}$, $B:\mathcal{N}\to \mathcal{X}$ be  maps. Let  $x\in \mathcal{M} \cap \mathcal{N}$  and $a, b \in \mathbb{C}$.  Then 
	\begin{align*}
		\Delta _x(A, a)^p+	\Delta _x(B,b)^p&\geq 	\frac{1}{2^{p-1}}\left(|\phi((A+B)x))|^p+C|\phi((A-B)x)|^p\right), \quad \forall \phi  \in \mathcal{X}^*\\
		& \quad \text{ satisfying } \|\phi\|\leq 1, \phi(x) =0.
	\end{align*}	
\end{theorem}
\begin{corollary}
Let $\mathcal{X}$ be  $p$-LWP(C).	Let 	$\mathcal{M}, \mathcal{N}\subseteq \mathcal{X}$ be  subsets and $A:\mathcal{M}\to \mathcal{X}$, $B:\mathcal{N}\to \mathcal{X}$ be  maps. Let  $x\in \mathcal{M} \cap \mathcal{N}$  and $a, b \in \mathbb{C}$ be such that $a+b=1$.  Then 
\begin{align*}
	\Delta _x(A, a)^p+	\Delta _x(B,b)^p&\geq 	\frac{1}{2^{p-1}}\left(|\phi((A+B)x))|^p+C|\phi((A-B)x)|^p\right), \quad \forall \phi  \in \mathcal{X}^\#\\
	& \quad \text{ satisfying } \|\phi\|_{\text{Lip}_0}\leq 1, \phi(x) =0.
\end{align*}		
\end{corollary}
We now proceed to derive nonlinear uncertainty principle for a class of Banach spaces (at present, we don't know it for arbitrary Banach spaces). Recall that a Banach space $\mathcal{X}$ is said to be of Type-p, $p\in [1,2]$ \cite{ALBIACKALTON} if there exists a constant $C>0$ satisfying following:  For every $n \in \mathbb{N}$, 
\begin{align}\label{TYPE}
	\left(\frac{1}{2^n}\sum_{(\varepsilon_j)_{j=1}^n\in \{-1,1\}^n}\left\|\sum_{j=1}^{n}\varepsilon_j x_j\right\|^p\right)^\frac{1}{p}\leq C\left(\sum_{j=1}^{n}\|x_j\|^p\right)^\frac{1}{p}, \quad \forall x_1, \dots, x_n \in \mathcal{X}.
\end{align}
In this case, we define the Type-p constant of  $\mathcal{X}$ as 
\begin{align*}
	T_p(\mathcal{X})\coloneqq \inf\left\{ C: C \text{ satisfies Inequality } (\ref{TYPE})\right\}.
\end{align*}
\begin{theorem}
	Let $\mathcal{X}$ be a Banach space of Type-p.  Let 	$\mathcal{M}, \mathcal{N}\subseteq \mathcal{X}$ be  subsets and $A:\mathcal{M}\to \mathcal{X}$, $B:\mathcal{N}\to \mathcal{X}$ be  maps. Let  $x\in \mathcal{M} \cap \mathcal{N}$  and $a, b \in \mathbb{C}$.  Then 
\begin{align*}
		\Delta _x(A, a)^p+	\Delta _x(B,b)^p&\geq 	\frac{1}{2T_p(\mathcal{X})^p}\left(|\phi((A+B)x))|^p+|\phi((A-B)x)|^p\right), \quad \forall \phi  \in \mathcal{X}^*\\
		& \quad \text{ satisfying } \|\phi\|\leq 1, \phi(x) =0.
	\end{align*}	
\end{theorem}
\begin{proof}
	Using the definition of Type-p, we get 
\begin{align*}
	\Delta _x(A, a)^p+	\Delta _x(B,b)^p&=\|Ax-ax \|^p+\|Bx-bx \|^p\\
	&\geq \frac{1}{2T_p(\mathcal{X})^p} \left(\|(Ax-ax)+(Bf-bf)\|^p+\|(Ax-ax)-(Bx-bx)\|^p\right)\\
	&=\frac{1}{2T_p(\mathcal{X})^p} \left(\|(A+B)x-(a+b)x\|^p+\|(A-B)x-(a-b)x\|^p\right)\\
	&\geq \frac{1}{2T_p(\mathcal{X})^p}\|\phi\|\left(\|(A+B)x-(a+b)x\|^p+\|(A-B)x-(a-b)x\|^p\right)\\
	&\geq \frac{1}{2T_p(\mathcal{X})^p}\left(|\phi((A+B)x)-\phi((a+b)x)|^p+|\phi((A-B)x)-\phi((a-b)x)|^p\right)\\
	&=\frac{1}{2T_p(\mathcal{X})^p}\left(|\phi((A+B)x))|^p+|\phi((A-B)x)|^p\right).
\end{align*}
\end{proof}
\begin{corollary}
	Let $\mathcal{X}$ be a Banach space of Type-p.  Let 	$\mathcal{M}, \mathcal{N}\subseteq \mathcal{X}$ be  subsets and $A:\mathcal{M}\to \mathcal{X}$, $B:\mathcal{N}\to \mathcal{X}$ be  maps. Let  $x\in \mathcal{M} \cap \mathcal{N}$  and $a, b \in \mathbb{C}$ be such that $a+b=1$.  Then 
\begin{align*}
	\Delta _x(A, a)^p+	\Delta _x(B,b)^p&\geq 	\frac{1}{2T_p(\mathcal{X})^p}\left(|\phi((A+B)x))|^p+|\phi((A-B)x)|^p\right), \quad \forall \phi  \in \mathcal{X}^\#\\
	& \quad \text{ satisfying } \|\phi\|_{\text{Lip}_0}\leq 1, \phi(x) =0.
\end{align*}		
\end{corollary}
Note that we can derive following  results if we won't bother about power $p$.
\begin{theorem}
	Let $\mathcal{X}$ be a Banach space.  Let 	$\mathcal{M}, \mathcal{N}\subseteq \mathcal{X}$ be  subsets and $A:\mathcal{M}\to \mathcal{X}$, $B:\mathcal{N}\to \mathcal{X}$ be  maps. Let  $x\in \mathcal{M} \cap \mathcal{N}$  and $a, b \in \mathbb{C}$.  Then 
\begin{align*}
	\Delta _x(A, a)+	\Delta _x(B,b)&\geq 	|\phi((A+B)x))|, \quad \forall \phi  \in \mathcal{X}^* \text{ satisfying } \|\phi\|\leq 1, \phi(x) =0.
\end{align*}		
\end{theorem}
\begin{proof}
	By directly applying triangle inequality
	\begin{align*}
		\Delta _x(A, a)+	\Delta _x(B,b)&=\|Ax-ax\|+\|Bx-bx\|\geq \|Ax-ax+Bx-bx\|\\
		&\geq 	 \|\phi\|\|Ax-ax+Bx-bx\|\geq |\phi((A+B)x)-\phi((a+b)x))|\\
		&=|\phi((A+B)x)|.
	\end{align*} 
\end{proof}
\begin{corollary}
	Let $\mathcal{X}$ be a Banach space.  Let 	$\mathcal{M}, \mathcal{N}\subseteq \mathcal{X}$ be  subsets and $A:\mathcal{M}\to \mathcal{X}$, $B:\mathcal{N}\to \mathcal{X}$ be  maps. Let  $x\in \mathcal{M} \cap \mathcal{N}$  and $a, b \in \mathbb{C}$ be such that $a+b=1$.  Then 
	\begin{align*}
		\Delta _x(A, a)+	\Delta _x(B,b)&\geq 	|\phi((A+B)x))|, \quad \forall \phi  \in \mathcal{X}^\#  \text{ satisfying } \|\phi\|_{\text{Lip}_0}\leq 1, \phi(x) =0.
	\end{align*}		
\end{corollary}

 \bibliographystyle{plain}
 \bibliography{reference.bib}

\end{document}